\documentclass[12pt,a4paper,twoside]{article}

\title{\sc On the Maxwell Inequalities\\ for Bounded and Convex Domains}
\def\shorttitle{On the Maxwell Inequalities for Bounded and Convex Domains}
\def\pauthor{Dirk Pauly}

\def\mylabelonoff{off}
\def\allowdisbrk{no}

\usepackage{a4,exscale,ifthen,amsfonts,amssymb,amsmath,amscd,graphicx}
\usepackage[english]{babel}
\usepackage[mathscr]{eucal}

\author{{\sf\pauthor}}
\numberwithin{equation}{section}

\newenvironment{acknow}{{\vspace*{1cm}\noindent\bf Acknowledgements }}{}
\newcommand{\bewboxw}{\mbox{}\hfill $\square$ \\}

\newenvironment{proof}{{\noindent\bf Proof }}{\bewboxw}

\newcommand{\keywords}[1]{{\noindent\bf Key Words }#1}

\ifthenelse{\equal{\mylabelonoff}{on}}
{\newcommand{\mylabel}[1]{\label{#1}\fbox{{\rm #1}}}}{\newcommand{\mylabel}[1]{\label{#1}\makebox[0mm][]{}}}
\ifthenelse{\equal{\allowdisbrk}{yes}}
{\allowdisplaybreaks}{}

\usepackage[mathscr]{eucal}
\usepackage{color}

\def\rc{\color{red}}

\newcommand{\nz}{\mathbb{N}}

\newcommand{\rz}{\mathbb{R}}

\newcommand{\rt}{\rz^3}

\newcommand{\na}{\nabla}

\DeclareMathOperator{\rot}{rot}

\renewcommand{\div}{\operatorname{div}}

\newcommand{\om}{\Omega}

\def\set#1#2{\{#1\,:\,#2\}}

\newcommand{\cp}{c_{\mathtt{p}}}
\newcommand{\cpc}{c_{\mathtt{p},\circ}}

\newcommand{\cmt}{c_{\mathtt{m,t}}}

\newcommand{\cmn}{c_{\mathtt{m,n}}}

\def\cR{\mathcal{R}}

\DeclareMathOperator{\Lebesgue}{\mathsf{L}}
\newcommand{\Lgen}[2]{\Lebesgue^{#1}_{#2}}

\def\Lt{\Lgen{2}{}}

\DeclareMathOperator{\Sobolev}{\mathsf{H}}
\newcommand{\Hgen}[3]{\overset{#3}{\Sobolev}{}^{#1}_{#2}}
\def\Ho{\Hgen{1}{}{}}

\def\Hoc{\Hgen{1}{}{\circ}}

\newcommand{\scp}[2]{\left\langle#1,#2\right\rangle}
\newcommand{\scps}[2]{\langle#1,#2\rangle}

\newcommand{\norm}[1]{\left|\normdst\left|#1\right|\normdst\right|}

\newtheorem{lem}{Lemma}

\newtheorem{theo}[lem]{Theorem}

\newtheorem{rem}[lem]{Remark}

\DeclareMathOperator{\rotspace}{\mathsf{R}}
\newcommand{\rotgen}[2]{\overset{#2}{\rotspace}{}_{#1}}

\DeclareMathOperator{\divspace}{\mathsf{D}}
\newcommand{\divgen}[2]{\overset{#2}{\divspace}{}_{#1}}

\newcommand{\ho}{\Ho}
\newcommand{\hoc}{\Hoc}

\newcommand{\normos}[1]{|#1|} 

\DeclareMathOperator{\diam}{diam}
\newcommand{\diamom}{\diam(\om)}

\newcommand{\rn}{\mathbb{R}}
\renewcommand{\rt}{\rn^3}

\renewcommand{\r}{\rotgen{}{}}
\renewcommand{\rc}{\rotgen{}{\circ}}
\renewcommand{\rz}{\rotgen{0}{}}
\newcommand{\rcz}{\rotgen{0}{\circ}}
\newcommand{\rcal}{\cR}
\newcommand{\rcalc}{\overset{\circ}{\cR}}
\renewcommand{\d}{\divgen{}{}}
\newcommand{\dc}{\divgen{}{\circ}}
\newcommand{\dz}{\divgen{0}{}}
\newcommand{\dcz}{\divgen{0}{\circ}}

\renewcommand{\norm}[1]{\normos{#1}}
\renewcommand{\scp}[2]{\scps{#1}{#2}}

\begin{document}

\maketitle{}

\begin{abstract}
For a bounded and convex domain in three dimensions we show that 
the Maxwell constants are bounded from below and above by Friedrichs' and Poincar\'e's constants.\\
\end{abstract}

\tableofcontents

\section{Introduction}

Let $\om\subset\rt$ be a bounded and convex domain.
It is well known that, e.g., by Rellich's selection theorem
using standard indirect arguments, 
the Poincar\'e\footnote{The estimate \eqref{poincarehoc}
is often called Friedrichs'/Steklov inequality as well.} inequalities
\begin{align}
\mylabel{poincarehoc}
\exists\,\cpc&>0&
\forall\,u&\in\hoc&
\norm{u}&\leq\cpc\norm{\na u},\\
\mylabel{poincareho}
\exists\,\cp&>0&
\forall\,u&\in\ho\cap\rn^{\bot}
&\norm{u}&\leq\cp\norm{\na u}
\end{align}
hold. Here, $\cpc$ and $\cp$ are the Poincar\'e constants, which satisfy 
$$0<\cpc=1/\sqrt{\lambda_{1}}<1/\sqrt{\mu_{2}}=\cp,$$
where $\lambda_{1}$ is the first Dirichlet 
and $\mu_{2}$ the second Neumann eigenvalue of the Laplacian.
By $\scp{\,\cdot\,}{\,\cdot\,}$ and $\norm{\,\cdot\,}$
we denote the standard inner product and induced norm
in $\Lt$ and we will write the usual $\Lt$-Sobolev spaces
as $\ho$ and $\hoc$, the latter is defined as the closure in $\ho$
of smooth and compactly supported test functions.
All spaces and norms are defined on $\om$.
Moreover, we introduce the standard Sobolev spaces 
for the rotation and divergence by $\r$ and $\d$.
As before, we will denote the closures of test vector fields
in the respective graph norms by $\rc$ and $\dc$.
An index zero at the lower right corner of the latter spaces 
indicates a vanishing derivative, e.g.,
$$\rz:=\set{E\in\r}{\rot E=0},\quad\dcz:=\set{E\in\dc}{\div E=0}.$$
As $\om$ is convex, it is especially simply connected 
and has got a connected boundary. Hence, the Neumann and Dirichlet fields
of $\om$ vanish, i.e., $\rz\cap\dcz=\rcz\cap\rz=\{0\}$.
By the Maxwell compactness properties, i.e.,
the compactness of the two embeddings
$$\rc\cap\d\hookrightarrow\Lt,\quad\r\cap\dc\hookrightarrow\Lt,$$
(and again by a standard indirect argument) the Maxwell inequalities
\begin{align}
\mylabel{maxestelecconv}
\exists\,\cmt&>0&
\forall\,E&\in\rc\cap\d&
\norm{E}&\leq\cmt\big(\norm{\rot E}^2+\norm{\div E}^2\big)^{1/2},\\
\mylabel{maxestmagconv}
\exists\,\cmn&>0&
\forall\,H&\in\r\cap\dc&
\norm{H}&\leq\cmn\big(\norm{\rot H}^2+\norm{\div H}^2\big)^{1/2}
\end{align}
hold. To the best of the author's knowledge, 
general bounds for the Maxwell constants $\cmt$ and $\cmn$ are missing.
On the other hand, at least estimates for $\cmt$ and $\cmn$ from above are very important
from the point of view of applications, such as preconditioning
or a priori and a posteriori error estimation for numerical methods.

In the paper at hand we will prove that
\begin{align}
\mylabel{cmcpintro}
\cpc\leq\cmt\leq\cmn=\cp\leq\diamom/\pi
\end{align}
holds true. We note that \eqref{cmcpintro}
is already well known in two dimensions,
even for general Lipschitz domains $\om\subset\rn^{2}$
(except of the last inequality),
but new in three dimensions.
Furthermore, the last inequality in \eqref{cmcpintro}
has been proved in the famous paper of Payne and Weinberger \cite{payneweinbergerpoincareconvex},
where also the optimality of the estimate was shown.
This paper contains a small mistake, which has been corrected in \cite{bebendorfpoincareconvex}.

\section{Results and Proofs}

We start with an inequality for irrotational fields.

\begin{lem}
\mylabel{lemNarbdiv}
For all $E\in\na\hoc\cap\d$ and all $H\in\na\ho\cap\dc$
$$\norm{E}\leq\cpc\norm{\div E},\quad\norm{H}\leq\cp\norm{\div H}.$$
\end{lem}

\begin{proof}
Let $\varphi\in\hoc$ with $E=\na\varphi$. By \eqref{poincarehoc} we get
$$\norm{E}^2
=\scp{E}{\na\varphi}
=-\scp{\div E}{\varphi}
\leq\norm{\div E}\norm{\varphi}
\leq\cpc\norm{\div E}\norm{\na\varphi}
=\cpc\norm{\div E}\norm{E}.$$
Let $\varphi\in\ho$ with $H=\na\varphi$ and $\varphi\bot\rn$. 
Since $H\in\dc$ and by \eqref{poincareho} we obtain
$$\norm{H}^2
=\scp{H}{\na\varphi}
=-\scp{\div H}{\varphi}
\leq\norm{\div H}\norm{\varphi}
\leq\cp\norm{\div H}\norm{\na\varphi}
=\cp\norm{\div H}\norm{H},$$
completing the proof.
\end{proof}

\begin{rem}
\mylabel{remNarbdiv}
Clearly, Lemma \ref{lemNarbdiv} extends to arbitrary 
Lipschitz domains $\om\subset\rn^{N}$, $N\in\nz$.
\end{rem}

As usual in the theory of Maxwell's equations, we need another crucial tool,
the Helmholtz decompositions of vector fields into irrotational and solenoidal vector fields.
For convex domains, these decompositions are very simple. We have
\begin{align}
\mylabel{helmdecoconvex}
\Lt=\na\hoc\oplus\rot\r,\quad\Lt=\na\ho\oplus\rot\rc,
\end{align}
where $\oplus$ denotes the orthogonal sum in $\Lt$.
We note
$$\rcz=\na\hoc,\quad\rz=\na\ho,\quad\dz=\rot\r,\quad\dcz=\rot\rc.$$
Moreover, with
$$\rcalc:=\rc\cap\rot\r,\quad\rcal:=\r\cap\rot\rc$$
we have
\begin{align}
\mylabel{helmdecoconvexaugmented}
\rc=\na\hoc\oplus\rcalc,\quad\r=\na\ho\oplus\rcal
\end{align}
and see
$$\rot\rc=\rot\rcalc,\quad\rot\r=\rot\rcal.$$
We note that all occurring spaces of range-type are closed subspaces of $\Lt$,
which follows immediately by the estimates \eqref{poincarehoc}-\eqref{maxestmagconv}.
More details about the Helmholtz decompositions can be found e.g. in \cite{leisbook}.

To get similar inequalities for solenoidal vector fields as in Lemma \ref{lemNarbdiv} we 
need a crucial lemma from \cite[Theorem 2.17]{amrouchebernardidaugegiraultvectorpot},
see also \cite{saranenineqfried,grisvardbook,giraultraviartbook,costabelcoercbilinMax} 
for related partial results.

\begin{lem}
\mylabel{french}
Let $E$ belong to $\rc\cap\d$ or $\r\cap\dc$. Then $E\in\ho$ and
\begin{align}
\mylabel{frenchformula}
\norm{\na E}^2\leq\norm{\rot E}^2+\norm{\div E}^2.
\end{align}
\end{lem}

We emphasize that for $E\in\hoc$ and any domain $\om\subset\rt$
\begin{align}
\mylabel{frenchformulaequal}
\norm{\na E}^2=\norm{\rot E}^2+\norm{\div E}^2
\end{align}
holds since $-\Delta=\rot\rot-\na\div$.
This formula is no longer valid if $E$ has just the tangential
or normal boundary condition but for convex domains 
the inequality \eqref{frenchformula} remains true.

\begin{lem}
\mylabel{lemNThreerot}
For all vector fields $E$ in $\rc\cap\rot\r$ or $\r\cap\rot\rc$
$$\norm{E}\leq\cp\norm{\rot E}.$$
\end{lem}

\begin{proof}
Let $E\in\rot\r=\rot\rcal$ and $\Phi\in\rcal$ with $\rot\Phi=E$.
Then $\Phi\in\ho$ by Lemma \ref{french} since $\rcal=\r\cap\dcz$.
Moreover, $\Phi=\rot\Psi$ can be represented by some $\Psi\in\rc$.
Hence, for any constant vector $a\in\rt$ we have $\scp{\Phi}{a}=0$. 
Thus, $\Phi$ belongs to $\ho\cap(\rt)^{\bot}$.
Then, since $E\in\rc$ and by Lemma \ref{french} we get
$$\norm{E}^2
=\scp{E}{\rot\Phi}
=\scp{\rot E}{\Phi}
\leq\norm{\rot E}\norm{\Phi}
\leq\cp\norm{\rot E}\norm{\na\Phi}
\leq\cp\norm{\rot E}\norm{\underbrace{\rot\Phi}_{=E}}.$$
If $E\in\rot\rc$ there exists $\Phi\in\rc$ with $\rot\Phi=E$. 
Using \eqref{helmdecoconvexaugmented} we decompose 
$$E=E_{0}+E_{\rot}\in\rz\oplus\rcal.$$
Then, $\rot E_{\rot}=\rot E$ and
again by Lemma \ref{french} we see $E_{\rot}\in\ho$.
Let $a\in\rt$ such that $E_{\rot}-a\in\ho\cap(\rt)^{\bot}$.
Since $\Phi\in\rc$, $\scp{\rot\Phi}{H_{0}}$ and $\scp{\rot\Phi}{a}$ vanish.
By Lemma \ref{french}
$$\norm{E}^2
=\scp{\rot\Phi}{E}
=\scp{\underbrace{\rot\Phi}_{=E}}{E_{\rot}-a}
\leq\norm{E}\norm{E_{\rot}-a}\\
\leq\cp\norm{E}\norm{\na E_{\rot}}
\leq\cp\norm{E}\norm{\underbrace{\rot E_{\rot}}_{=\rot E}}$$
holds, which completes the proof.
\end{proof}

\begin{rem}
\mylabel{lemNThreerottwod}
It is well known that Lemma \ref{lemNThreerot} 
holds in two dimensions for any Lipschitz domain $\om\subset\rn^2$.
This follows immediately from Lemma \ref{lemNarbdiv}
if we take into account that in two dimensions the rotation $\rot$
is given by the divergence $\div$ after $90^{\circ}$-rotation 
of the vector field to which it is applied.
\end{rem}

\begin{theo}
\mylabel{maintheo}
For all vector fields $E\in\rc\cap\d$ and $H\in\r\cap\dc$
$$\norm{E}^2
\leq\cpc^2\norm{\div E}^2+\cp^2\norm{\rot E}^2,\quad
\norm{H}^2
\leq\cp^2\norm{\div H}^2+\cp^2\norm{\rot H}^2$$
hold, i.e., $\cmt,\cmn\leq\cp$.
Moreover, $\cpc\leq\cmt\leq\cmn=\cp\leq\diam(\om)/\pi$.
\end{theo}

\begin{proof}
By the Helmholtz decomposition \eqref{helmdecoconvex} we have
$$\rc\cap\d\ni E=E_{\na}+E_{\rot}\in\na\hoc\oplus\rot\r$$
with $E_{\na}\in\na\hoc\cap\d$ and $E_{\rot}\in\rc\cap\rot\r$ as well as
$\div E_{\na}=\div E$ and $\rot E_{\rot}=\rot E$.
By Lemma \ref{lemNarbdiv} and Lemma \ref{lemNThreerot} and orthogonality we obtain
$$\norm{E}^2=\norm{E_{\na}}^2+\norm{E_{\rot}}^2
\leq\cpc^2\norm{\div E}^2+\cp^2\norm{\rot E}^2.$$
Similarly we have
$$\r\cap\dc\ni H=H_{\na}+H_{\rot}\in\na\ho\oplus\rot\rc$$
with $H_{\na}\in\na\ho\cap\dc$ and $H_{\rot}\in\r\cap\rot\rc$ as well as
$\div H_{\na}=\div H$ and $\rot H_{\rot}=\rot H$.
As before,
$$\norm{H}^2=\norm{H_{\na}}^2+\norm{H_{\rot}}^2
\leq\cp^2\norm{\div H}^2+\cp^2\norm{\rot H}^2.$$
This shows the upper bounds.
For the lower bounds, let $\lambda_{1}$ be the first Dirichlet eigenvalue
of the negative Laplacian $-\Delta$, i.e.,
$$\frac{1}{\cpc^2}
=\lambda_{1}
=\inf_{0\neq u\in\hoc}\frac{\norm{\na u}^2}{\norm{u}^2},$$
and let $u\in\hoc$ be an eigenfunction to $\lambda_{1}$.
Note that $u$ satisfies 
$$\forall\,\varphi\in\hoc\quad\scp{\na u}{\na\varphi}=\lambda_{1}\scp{u}{\varphi}.$$
Then $0\neq E:=\na u\in\na\hoc\cap\d=\rcz\cap\d$ 
and $-\div E=-\div\na u=\lambda_{1}u$.
By \eqref{maxestelecconv} and \eqref{poincarehoc} we have
\begin{align*}
\norm{E}
&\leq\cmt\norm{\div E}
=\cmt\lambda_{1}\norm{u}
\leq\cmt\lambda_{1}\cpc\norm{\na u}
=\frac{\cmt}{\cpc}\norm{E},
\end{align*}
yielding $\cpc\leq\cmt$.
Now, let $\mu_{2}$ be the second Neumann eigenvalue
of the negative Laplacian $-\Delta$, i.e.,
$$\frac{1}{\cp^2}
=\mu_{2}
=\inf_{0\neq u\in\ho\cap\rn^{\bot}}\frac{\norm{\na u}^2}{\norm{u}^2},$$
and let $u\in\ho\cap\rn^{\bot}$ be an eigenfunction to $\mu_{2}$.
Note that $u$ satisfies 
$$\forall\,\varphi\in\ho\cap\rn^{\bot}\quad\scp{\na u}{\na\varphi}=\mu_{2}\scp{u}{\varphi}$$
and that this relation holds even for all $\varphi\in\ho$.
Then $0\neq H:=\na u\in\na\ho\cap\dc=\rz\cap\dc$ 
and satisfies $-\div H=-\div\na u=\mu_{2}u$.
By \eqref{maxestmagconv} and \eqref{poincareho} we have
\begin{align*}
\norm{H}
&\leq\cmn\norm{\div H}
=\cmn\mu_{2}\norm{u}
\leq\cmn\mu_{2}\cp\norm{\na u}
=\frac{\cmn}{\cp}\norm{H},
\end{align*}
yielding $\cp\leq\cmn$ and completing the proof.
\end{proof}

\begin{rem}
\mylabel{eigenvalues}
\begin{itemize}
\item[\bf(i)] 
It is unclear but most probable that $\cpc<\cmt<\cmn=\cp$ holds.
In forthcoming publications \cite{paulymaxconst1,paulymaxconst2}
we will show more and sharper estimates on the Maxwell constants,
showing additional and sharp relations 
between the Maxwell and the Poincar\'e/Friedrichs/Steklov constants.
\item[\bf(ii)] 
Our results extend also to all polyhedra which allow the 
$\ho$-regularity of the Maxwell spaces $\rc\cap\d$ and $\r\cap\dc$
or to domains whose boundaries consist of combinations of convex boundary parts 
and polygonal parts which allow the $\ho$-regularity.
Is is shown in \cite[Theorem 4.1]{costabelcoercbilinMax} 
that \eqref{frenchformula}, even \eqref{frenchformulaequal}, 
still holds for all $E\in\ho\cap\rc$ or $E\in\ho\cap\dc$ 
if $\om$ is a polyhedron\footnote{The crucial point is 
that the unit normal is piecewise constant and hence the curvature is zero.}.
We note that even some non-convex polyhedra admit the $\ho$-regularity of the Maxwell spaces
depending on the angle of the corners, which are not allowed to by too pointy.
\item[\bf(iii)] 
Looking at the proof, the lower bounds $\cpc\leq\cmt$ and $\cp\leq\cmn$ 
remain true in more general situations, i.e., 
for bounded Lipschitz\footnote{The Lipschitz assumption
can also be weakened. It is sufficient that $\om$ admits the Maxwell compactness properties.}  
domains $\om\subset\rt$.
\end{itemize}
\end{rem}

\begin{acknow}
The author is deeply indebted to Sergey Repin not only for bringing his attention
to the problem of the Maxwell constants in 3D.
Moreover, the author wants to thank Sebastian Bauer und Karl-Josef Witsch 
for nice and  deep discussions.
\end{acknow}

\bibliographystyle{plain} 
\bibliography{/Users/paule/Library/texmf/tex/TeXinput/bibtex/paule}

\end{document}